\documentclass[11pt]{article}

\usepackage{mathtools}
\usepackage{amssymb}
\usepackage{amsthm}
\usepackage{xcolor}
\usepackage{tikz}
\usepackage{fullpage}

\DeclareSymbolFont{bbold}{U}{bbold}{m}{n}
\DeclareSymbolFontAlphabet{\mathbbold}{bbold}

\newcommand{\N}{\mathbb{N}}
\newcommand{\Z}{\mathbb{Z}}
\newcommand{\Q}{\mathbb{Q}}
\newcommand{\R}{\mathbb{R}}
\newcommand{\C}{\mathbb{C}}
\newcommand{\K}{\mathbb{K}}

\newcommand{\1}{\mathbbold{1}}

\newcommand{\loc}{\mathrm{loc}}
\newcommand{\lu}{\mathrm{unif}}
\newcommand{\unif}{\mathrm{unif}}
\newcommand{\Mloc}{\mathcal{M}_{\loc}(\R)}
\newcommand{\M}{\mathcal{M}_{\loc,\lu}(\R)}

\newcommand{\MR}{\mathcal{AM}(\R)}

\newcommand{\Dir}{{\rm D}}
\newcommand{\Neu}{{\rm N}}

\makeatletter
\newcommand{\uD}{\u@DN{\Dir}}
\newcommand{\uN}{\u@DN{\Neu}}
\newcommand{\u@DN}[1]{u_{#1\mkern-1mu}}
\makeatother

\DeclareMathOperator{\spt}{spt}
\DeclareMathOperator{\diam}{diam}

\renewcommand{\Re}{\operatorname{Re}}

\DeclareMathOperator{\tr}{tr}
\DeclareMathOperator{\Per}{Per}

\newcommand{\from}{\colon}

\let\phi\varphi
\let\le\leqslant
\let\leq\leqslant
\let\ge\geqslant
\let\geq\geqslant

\makeatletter
\def\@row#1,{#1\@ifnextchar;{\@gobble}{&\@row}}
\def\@matrix{%
    \expandafter\@row\my@arg,;%
    \@ifnextchar({\\ \get@in@paren{\@matrix}}{\after@matrix}%
    }
\def\matrixtype#1#2#3{%
    \ifmmode\def\after@matrix{\end{#2}\right#3}%
    \else\def\after@matrix{\end{#2}\right#3$}$\fi
    \left#1\begin{#2}\get@in@paren{\@matrix}%
    }
\def\@column#1,{#1\@ifnextchar;{\@gobble}{\\ \@column}}
\newcommand\vect{}
\def\svect(#1){\left(\begin{smallmatrix}\@column#1,;\end{smallmatrix}\right)}
\def\vect{\get@in@paren{\@vect}}
\def\@vect{\left(\begin{matrix}\expandafter\@column\my@arg,;\end{matrix}\right)}
\def\get@in@paren#1({\def\my@arg{}\def\my@rest{}\def\after@get{#1}\get@arg}
\let\e@a\expandafter
\def\get@arg#1){\e@a\kl@test\my@rest#1(;}
\def\kl@test#1(#2;{\e@a\def\e@a\my@arg\e@a{\my@arg#1}%
                   \ifx:#2:\let\my@exec\after@get
                   \else\let\my@exec\get@arg
                        \e@a\def\e@a\my@arg\e@a{\my@arg(}%
                        \def@rest#2;%
                   \fi\my@exec}
\def\def@rest#1(;{\def\my@rest{#1\kl@zu}}
\def\kl@zu{)}

\makeatletter
\newcommand\MyPairedDelimiter{%
  \@ifstar{\My@Paired@Delimiter{{}}}
          {\My@Paired@Delimiter{}}%
}
\newcommand\My@Paired@Delimiter[4]{%
  \newcommand#2{%
    \@ifstar{\start@PD{#1}{\delimitershortfall=-1sp}{#3}{#4}}
            {\start@PD{#1}{}{#3}{#4}}%
  }%
}
\newcommand\start@PD[5]{%
  #1\mathopen{\mathpalette\put@delim@helper{\put@delim{#2}{#3}{.}{#5}}}%
  #5%
  \mathclose{\mathpalette\put@delim@helper{\put@delim{#2}{.}{#4}{#5}}}%
}
\newcommand\put@delim@helper[2]{%
  \hbox{$\m@th\nulldelimiterspace=0pt #2#1$}%
}
\newcommand\put@delim[5]{%
  \setbox\z@\hbox{$\m@th#5{#4}$}%
  \setbox\tw@\null
  \ht\tw@\ht\z@ \dp\tw@\dp\z@
  #1#5%
  \left#2\box\tw@\right#3%
}

\makeatother
\MyPairedDelimiter*{\abs}{\lvert}{\rvert}
\MyPairedDelimiter*{\norm}{\lVert}{\rVert}
\MyPairedDelimiter{\set}{\{}{\}}

\newcommand\llim{
\mathchoice{\vcenter{\hbox{${\scriptstyle{-}}$}}}
{\vcenter{\hbox{$\scriptstyle{-}$}}}
{\vcenter{\hbox{$\scriptscriptstyle{-}$}}}
{\vcenter{\hbox{$\scriptscriptstyle{-}$}}}}

\theoremstyle{plain} 
\newtheorem{theorem}{Theorem}[section]
\newtheorem{corollary}[theorem]{Corollary}
\newtheorem{lemma}[theorem]{Lemma}
\newtheorem{proposition}[theorem]{Proposition}

\theoremstyle{definition}
\newtheorem{example}[theorem]{Example}
\newtheorem*{definition}{Definition}
\newtheorem{remark}[theorem]{Remark}

\usepackage{enumitem}

\setenumerate[1]{nolistsep} 
\setenumerate[2]{nolistsep} 

\newcommand{\Hmm}[1]{\leavevmode{\marginpar{\tiny%
$\hbox to 0mm{\hspace*{-0.5mm}$\leftarrow$\hss}%
\vcenter{\vrule depth 0.1mm height 0.1mm width \the\marginparwidth}%
\hbox to 0mm{\hss$\rightarrow$\hspace*{-0.5mm}}$\\\relax\raggedright #1}}}

\begin{document}

\medmuskip=4mu plus 2mu minus 3mu
\thickmuskip=5mu plus 3mu minus 1mu
\belowdisplayshortskip=9pt plus 3pt minus 5pt

\title{On eigenvalue bounds for a general class of Sturm-Liouville operators}


\author{Christian Seifert}

\date{\today}

\maketitle

\begin{abstract}
  \noindent
  We consider Sturm-Liouville operators with measure-valued weight and potential, and positive, bounded diffusion coefficient which is bounded away from zero.
  By means of a local periodicity condition, which can be seen as a quantitative Gordon condition, 
  we prove a bound on eigenvalues for the corresponding operator in $L_p$, for $1\leq p<\infty$. 
  We also explain the sharpness of our quantitative bound, and provide an example for quasiperiodic operators.

  
\end{abstract}

\section{Introduction}
\label{sec:Introduction}

In this paper we study bounds on (and absence of) eigenvalues for (elliptic) Sturm-Liouville operators $H:=H_{p,\rho,a,\mu}$ in $L_p(\R,\rho)$ acting on $u$ as
\[Hu := \partial_\rho\Bigl( -au' + \int_0^{(\cdot)} u\,d\mu\Bigr).\]
Here, $p\in[1,\infty)$, $\rho$ is a non-negative locally finite periodic measure, 
$0\leq a\in L_\infty(\R)$ with $\frac{1}{a}\in L_\infty(\R)$ and $\mu$ is a real uniformly locally finite measure.
Such operators include classical Sturm-Liouville operators, continuum Schr\"odinger operators with (local) measures as potential, discrete Schr\"odinger operators and Jacobi matrices, 
providing a unified framework.

For fixed $p$ and $\rho$, we show quantitatively that $H$ does not have eigenvalues with small modulus,
provided for a sequence $(p_m)$ of periods tending to infinity the coefficents $a$ and $\mu$ restricted to $[-p_m,0]$, $[0,p_m]$ and $[p_m,2p_m]$ look very similar.
Such a condition is sometimes called Gordon-codition due to \cite{Gordon1976}, see also \cite{Gordon1986,Damanik2000,DamanikStolz2000,Damanik2004,Seifert2011,Seifert2012,SeifertVogt2014,Fillman2015},
for various situations. Note that in these references, almost exclusively the case of Schr\"odinger operators are treated (except for \cite{Fillman2015}, where CMV-matrices are considered),
and except of \cite{SeifertVogt2014} all results are qualitative.

The quantitative bound we provide is in general not sharp. 
However, we can derive a sharp bound by minor modifications (see also Section 6 of our previous treatment \cite{SeifertVogt2014} for details). 
Thus, this paper can be seen as a generalization of \cite{SeifertVogt2014}.

Our results can be applied to quasiperiodic coefficients where the ratio of the periods can be well-approximated by rational numbers (a so-called strong Liouville condition).
Such an assumption is typical in the treatment of one-dimensional quasicrystal models.

The paper is organised as follows. 
In Section \ref{sec:Sturm-Liouville_operators} we introduce the Sturm-Liouville operators we are dealing with and show how the special cases mentioned above can be derived.
Section \ref{sec:Solutions} deals with solutions of the eigenvalue equation, and provides some first estimates of solutions in terms of the coefficients. 
Here, we work with $L_\infty$-estimates for $\frac{1}{a}$ and a uniform local norm for $\mu$.
The following Section \ref{sec:differences_solutions} focusses on estimating differences of solutions in terms of differences of the coefficients, 
which is measured in $L_1$ for the diffusion coefficent and in a Wasserstein-type seminorm for the potential.
In the final Section \ref{sec:Gordons_Theorem} we state the precise condition for absence of eigenvalues, state and prove the eigenvalue bound, comment on the sharpness of it and provide the example.
In an appendix we include a Gronwall inequality suitable for our purpose.

\section{Sturm-Liouville operators with measure-valued coefficients}
\label{sec:Sturm-Liouville_operators}

Let $\K\in\set{\R,\C}$.
Let $\mathcal{B}(\R)$ denote the Borel $\sigma$-field on $\R$.
A mapping $\mu\from \set{B\in\mathcal{B}(\R);\; B\,\text{bounded}}\to \K$ is called a \emph{local measure} if
$\1_K\mu:=\mu(\cdot\cap K)$ is a (finite) $\K$-valued Radon measure for all compact subsets $K\subseteq \R$.
Then there exist a (unique) nonnegative Radon measure
$\nu$ on $\R$ and a measurable function $\sigma\from\R\to \K$ such that $\abs{\sigma}=1$ $\nu$-a.e.\ and
$\1_K\mu = \1_K \sigma\nu$ for all compact sets $K\subseteq \R$. The total variation of $\mu$ is defined
by $\abs{\mu}:=\nu$.
Let $\Mloc$ be the space of all local measures on $\R$.

A local measure $\mu\in \Mloc$ is called \emph{uniformly locally bounded} if
\[
  \norm{\mu}_\lu := \sup_{t\in\R} \abs{\mu}\bigl((t,t+1]\bigr) < \infty.
\]
Let $\M$ denote the space of all uniformly locally bounded local measures.
The space $\M$ naturally extends $L_{1,\loc,\unif}(\R)$ to measures.

\begin{remark}
  Let $\mu\in\Mloc$. Then the set $\set{t\in\R;\; \mu(\set{t}) \neq 0}$ of atoms of $\mu$ is at most countable.
\end{remark}

We say that $f\from\R\to \K$ is \emph{locally absolutely continuous} with respect to $\mu\in \Mloc$ if there exists $h\in L_{1,\loc}(\R,\abs{\mu})$ such that
\[f(t) = f(c) + \int_c^t h(s)\,d\mu(s) \quad(t\in \R),\]
for some $c\in \R$, where
\[
  \int_s^t \dots \, d\mu := \begin{cases}
                         \int_{(s,t]} \dots \, d\mu & \text{ if } t\ge s,\\
                        -\int_{(t,s]} \dots \, d\mu & \text{ if } t<s.
                       \end{cases}
\]
Then $h$ is the Radon-Nikodym derivative of $f$ with respect to $\mu$, which is uniquely defined in $L_{1,\loc}(\R,\abs{\mu})$. We will write
$\partial_\mu f := h$.
Furthermore, $f$ is then right-continuous and locally of bounded variation, so also the limits from the left exist everywhere.

\begin{remark}[jump heights]
  Let  $\mu\in \M$, $f\from \R\to \K$ be measurable. Assume $\partial_\rho f\in L_{1,\loc}(\R,\rho)$. Then
  \[f(t) = f(t\llim) + \partial_\rho f(t)\rho(\set{t}) \quad(t\in\R).\]
\end{remark}

\medskip

Let $0\leq \rho \in \M$, $\rho\neq 0$. Let
\[\Per(\rho):=\set{p\in\R\setminus\{0\};\; \rho(\cdot+p) = \rho}\]
be the set of periods of $\rho$. Note that $\rho$ is periodic if and only if $\Per(\rho)$ is an infinite set if and only if $\Per(\rho)\neq \varnothing$. 
Clearly, then the support $\spt\rho$ of $\rho$ is an infinite set.

Let $a\from\R\to\K$ be measurable and right-continuous, $\mu\in \M$. For $u\in W_{1,\loc}^1(\R)$ define $A_{a,\mu}u$ by
\[A_{a,\mu}u(t):= -(au')(t) + \int_0^t u(s)\,d\mu(s)\]
for a.a.\ $t\in \R$. Note that $A_{a,\mu}\in L_{1,\loc}(\R)$.
Define 
\begin{align*}
  D:=D_{\rho,a,\mu}(\R) & := \bigl\{u\in W_{1,\loc}^1(\R);\; A_{a,\mu}u\, \text{locally absolutely continuous w.r.t.\ $\rho$}\bigr\}.
\end{align*}
Note that for $u\in D$ also $au'$ is right-continuous and locally of bounded variation.

\bigskip

For the rest of that paper, let $0\leq \rho \in \M$, $\rho\neq 0$ be periodic, and write
\[\MR:=\set{(a,\mu);\; a\from\R\to [0,\infty),\, a,\tfrac{1}{a}\in L_\infty(\R),\; \mu\in\M\,\text{real}, \spt\mu\subseteq\spt\rho}.\]

\medskip

Let $p\in[1,\infty)$. For $(a,\mu)\in \MR$ we define the operator $H :=H_{p,\rho,a,\mu}$ in $L_p(\R,\rho)$ by
\begin{align*}
  D(H) & := \set{u\in L_p(\R,\rho);\; u\in D,\, \partial_\rho A_{a,\mu} u\in L_p(\R,\rho)},\\
  H u & := \partial_\rho A_{a,\mu} u.
\end{align*}
Note that by the reasoning in \cite[Sections 3 and 4]{EckhardtTeschl2013}, $H$ is indeed a densely defined operator in $L_p(\R,\rho)$. 

\begin{example}
  Let $r\in L_{1,\loc}(\R)$, $r>0$ a.e., $\rho := r\lambda$, where $\lambda$ is the Lebesgue measure on $\R$, $0\leq a\in L_{\infty}(\R)$ such that $\frac{1}{a}\in L_{\infty}(\R)$, 
  $q\in L_{1,\loc}(\R)$ real, $\mu:= q\lambda$. Then $(a,\mu)\in \MR$, and $H$ acts as
  \[Hu = \frac{1}{r}(-(au')' + qu),\]
  i.e.\ as a classical Sturm-Liouville operator.
\end{example}

\begin{example}
  Let $\rho := \lambda$, where $\lambda$ is the Lebesgue measure on $\R$, $a:=1$ , 
  $\mu\in\M$ real. Then $(a,\mu)\in \MR$, and $H$ acts as
  \[Hu = -u''+ u\mu,\]
  i.e.\ as a one-dimensional continuum Schr\"odinger operator with a local measure as potential.
\end{example}

\begin{example}
  Let $\rho := \delta_\Z:= \sum_{n\in\Z}\delta_n$, $(a_n)_{n\in\Z}$ in $(0,\infty)$ be bounded such that $(\frac{1}{a_n})$ is also bounded
  $a:=\sum_{n\in\Z} a_n \1_{[n,n+1)}$, $(b_n)_{n\in\Z}$ in $\R$, $\mu:=\sum_{n\in\Z} b_n\delta_n$. Then $(a,\mu)\in \MR$, and $H$ acts as
  \[Hu(n) = a_{n-1}\bigl(u(n)-u(n-1)\bigr) - a_n\bigl(u(n+1)-u(n)\bigr) + b_nu(n) \quad(n\in\Z),\]
  i.e.\ as a Jacobi operator.
\end{example}

\section{Solutions of the eigenvalue equation}
\label{sec:Solutions}

\begin{definition}
  Let $(a,\mu)\in\MR$, $z\in\C$. We say that $u\from \R\to\K$ is a \emph{solution} of 
  \[Hu = zu,\]
  if $u\in D$ and $\partial_\rho A_{a,\mu} u = zu$ in $L_{1,\loc}(\R,\rho)$.
\end{definition}
By \cite[Theorem 3.1]{EckhardtTeschl2013}, solutions exist and are uniquely defined by the values of $u$ and $au'$ at the same point $t\in\R$ (put differently, the space of solutions is two-dimensional).
Note that $u$ is a solution of $H_{p,\rho,a,\mu}u=zu$ if and only if $u$ is a solution of $H_{p,\rho,a,\mu-z\rho} u = 0$.
Furthermore, for real $\rho$, $a$, $\mu$ and $z$ also solutions $u$ of $Hu = zu$ can be chosen to be real.

\begin{remark}
  Let $(a,\mu)\in\MR$, $z\in\C$ and $u$ be a solution of $Hu=zu$. Then $au'$ is constant on every connected component of $\R\setminus \spt\rho$. 
  Indeed, $u$ satisfies, for some $c\in\R$,
  \[z \int_c^t u\,d\rho = A_{a,\mu}u(t) - A_{a,\mu}u(c) = -(au')(t) + (au')(c) + \int_c^t u\,d\mu \quad(t\in\R).\]
\end{remark}

\begin{definition}
  Let $(a,\mu)\in \MR$. For $s\in\R$ let $\uN(\cdot;s)$, $\uD(\cdot;s)$ are the solutions of $Hu = 0$ satisfying
  \begin{align*}
    \uN(s;s) & = 1 & \uD(s;s) & = 0 & \\
    (a\uN'(\cdot;s))(s) & = 0 & \bigl(a\uD'(\cdot;s)\bigr)(s) & = 1. & \\
  \end{align*}
  Then $\uN(\cdot;s)$ and $\uD(\cdot;s)$ are called \emph{Neumann} and \emph{Dirichlet} solution (with initial condition at $s$), respectively.
  For $s,t\in\R$ we denote by
  \[T_{a,\mu}(t,s) := \begin{pmatrix}
		\uN(t;s) & \uD(t;s)\\
		\bigl(a\uN'(\cdot;s)\bigr)(t) & \bigl(a\uD'(\cdot;s)\bigr)(t)
		\end{pmatrix}\]
  the \emph{transfer matrices} for the equation $Hu=0$.
\end{definition}

\begin{lemma}
  Let $(a,\mu)\in \MR$, $u\in D$.
  The following are equivalent:
  \begin{enumerate}
    \item
      $u$ is a solution of the equation $Hu = 0$.
    \item
      For $s,t\in\R$ we have
      \[\begin{pmatrix}
	  u(t) \\
	  (au')(t)
	\end{pmatrix} = T_{a,\mu}(t,s)\begin{pmatrix}
	  u(s) \\
	  (au')(s)
	\end{pmatrix}.\]
  \end{enumerate}
\end{lemma}

\begin{proof}
  ``(a)$\Rightarrow$(b)'': Fix $s,t\in\R$ and let
      \[\tilde{T}_{a,\mu}(t,s)\from \begin{pmatrix}
	  u(s) \\
	  (au')(s)
	\end{pmatrix}\mapsto \begin{pmatrix}
	  u(t) \\
	  (au')(t)
	\end{pmatrix},\]
	i.e.\ the mapping which shifts solutions (of the corresponding first order system) at $s$ to solutions at $t$.
	Then $\tilde{T}_{a,\mu}(t,s)$ is linear and can be represented by a matrix, which we will also denote by $\tilde{T}_{a,\mu}(t,s)$.
	By the initial conditions for the Neumann and Dirichlet solution we observe
	\begin{align*}
	\tilde{T}_{a,\mu}(t,s) & = \tilde{T}_{a,\mu}(t,s)\begin{pmatrix} 1 & 0\\0 & 1\end{pmatrix} = \tilde{T}_{a,\mu}(t,s) \begin{pmatrix}
		    \uN(s;s) & \uD(s;s)\\
		    \bigl(a\uN'(\cdot;s)\bigr)(s) & \bigl(a\uD'(\cdot;s)\bigr)(s)
		    \end{pmatrix}\\
		  & = \begin{pmatrix}
		    \uN(t;s) & \uD(t;s)\\
		    \bigl(a\uN'(\cdot;s)\bigr)(t) & \bigl(a\uD'(\cdot;s)\bigr)(t)
		    \end{pmatrix} = T_{a,\mu}(t,s).
	\end{align*}
  ``(b)$\Rightarrow$(a)'': Fix $s\in\R$. For $t\in\R$ we have
  \begin{align*}
    u(t) & = \uN(t;s)\cdot u(s) + \uD(t;s)\cdot (au')(s),\\
    (au')(t) & =  \bigl(a\uN'(\cdot;s)\bigr)(t)\cdot u(s) + \bigl(a\uD'(\cdot;s)\bigr)(t)\cdot (au')(s).
  \end{align*}
  Thus,
  \begin{align*}
    -(au')(t) + \int_s^t u(r)\,\mu(r) & = u(s) \bigg(- \bigl(a\uN'(\cdot;s)\bigr)(t) + \int_s^t \uN(r;s)\,d\mu(r)\bigg) \\
    & \quad + (au')(s)\bigg( -\bigl(a\uD'(\cdot;s)\bigr)(t) + \int_s^t \uD(r;s)\,d\mu(r)\bigg). 
  \end{align*}
  Differentiating with respect to $\rho$ yields
  \[Hu = u(s) H\uN(\cdot;s) + (au')(s) H\uD(\cdot;s) = u(s)\cdot 0 + (au')(s)\cdot 0 = 0.\]
  Hence, $u$ is a solution of $H u = 0$.	
\end{proof}

\begin{lemma}
  Let $(a,\mu)\in \MR$. Then $\det T_{a,\mu}(t,s) = 1$ for all $s,t\in\R$, and
  \[T_{a,\mu}(s,t) = T_{a,\mu}(t,s)^{-1} = \begin{pmatrix}
							  \bigl(a\uD(\cdot;s)'\bigr)(t) & -\uD(t;s)\\
							  -\bigl(a\uD(\cdot;s)'\bigr)(t) & \uN(t;s)
                                                         \end{pmatrix} \quad(s,t\in\R).\]
\end{lemma}

\begin{proof}
  By the Lagrange-identity, see \cite[Proposition 3.2]{EckhardtTeschl2013}, the determinant of the transfer matrices is constant. Thus, it equals $1$.
  The formula for the inverse matrix is then an immediate consequence.
\end{proof}

\begin{lemma}
\label{lem:conv_to_0}
  Let $u\from \R\to\K$ be measurable, right-continuous, $u\in L_p(\R,\rho)$. Assume that for all $r>0$ we have
  \[\abs{u(t+r)-u(t)}\to 0 \quad(\abs{t}\to \infty).\]
  Then $u(t)\to 0$ as $\abs{t}\to\infty$.
\end{lemma}

\begin{proof}
  Let $s\in \Per(\rho)$. Without loss of generality we may assume that $\rho((0,s]) = 1$ and $u\geq0$ (thanks to the reverse triangle inequality).
  
  Assume that $u(t)\not\to 0$ as $t\to\infty$. Then there exists $\delta>0$ and $(t_n)$ in $(0,\infty)$ such that
  $t_n\to\infty$ and $u(t_n)\geq \delta$ for all $n\in\N$.
  Since $u\in L_p(\R,\rho)$ we have $\norm{\1_{(t_n,t_n+s]}u}_{L_p(\R,\rho)} \to 0$. Passing to a subsequence we may assume that
  \[\norm{\1_{(t_n,t_n+s]}u}_{L_p(\R,\rho)} \leq 2^{-2n} \quad(n\in\N).\]
  By Markov's inequality, we observe
  \[\rho(\set{t\in(t_n,t_n+s];\; u(t)\geq 2^{-n}}) \leq \frac{\norm{\1_{(t_n,t_n+p]}u}_{L_p(\R,\rho)}^p}{2^{-np}} \leq 2^{-np}.\]
  Let $A_n:=\set{t\in(0,s];\; u(t_n+t)\geq 2^{-n}}$. Then $\rho(A_n) \leq 2^{-np}\leq 2^{-n}$, and therefore
  \[\rho(\cup_{n\geq 3} A_n) \leq \sum_{n\geq 3} \rho(A_n) \leq 2^{-2}<1.\]
  Hence, $G:=(0,s]\setminus (\cup_{n\geq 3} A_n)$ has positive $\rho$-measure and is therefore non-empty. Let $r\in G$. Then $u(t_n+r) < 2^{-n}$ for $n\geq 3$, and therefore
  \[\liminf_{n\to\infty} \abs{u(t_n+r)-u(t_n)}\geq \delta > 0. \qedhere\]  
\end{proof}

\begin{lemma}
\label{lem:tends_to_zero}
  Let $(a,\mu)\in \MR$, $u\in L_p(\R,\rho)$ a solution of $Hu=0$.
  Then $u(t)\to 0$ for $\abs{t}\to \infty$.
\end{lemma}

\begin{proof}
  This is a direct consequence of Lemma \ref{lem:conv_to_0}.  
\end{proof}

Lemma \ref{lem:tends_to_zero} states that eigenfunctions $u\in L_p(\R,\rho)$ of $H$ have to tend to $0$ at $\pm\infty$.

The next lemma establishes a control of the derivative of solutions by means of the solution itself.

\begin{lemma}
\label{lem:estimate_derivative}
  Let $(a,\mu)\in \MR$, $u$ be a real solution of $Hu=0$, $r>0$.
  Then there exists $C>0$ such that for all intervals $I=(\alpha,\beta]\subseteq \R$ with $\beta-\alpha=r$ we have
  \[\norm{(au')|_I}_\infty \leq C \norm{u|_I}_\infty.\]
  We can set
  \[C = C_{r,a,\mu} := \max\set{\frac{2\norm{a}_\infty}{r}, \frac{\norm{\mu}_\unif}{2}} + \lceil r \rceil \norm{\mu}_\unif .\]
\end{lemma}

\begin{proof}
  For $u=0$ the assertion is trivial. Hence, let $u\neq 0$.

  We first show that for an interval $I$ there exist $C>0$ and $s\in I$ such that $\abs{(au')(s)}\leq C \norm{u|_I}_\infty$.
  Assume this inequality does not hold. Then, for all $C>0$ we have $\abs{(au')(s)} > C \norm{u|_I}_\infty$ for all $s\in I$.
  Since $au'$ is real, and $(au')(s) - (au')(s\llim) = u(s)\mu(\set{s}) \leq \norm{\mu}_\unif \norm{u|_I}_\infty$, for $C>\tfrac{\norm{\mu}_\unif}{2}$ we obtain either
  $(au')(t) \geq C \norm{u|_I}_\infty$ for all $t\in I$ or $-(au')(t) \leq - C \norm{u|_I}_\infty$ for all $t\in I$.
  Since
  \[u(t) - u(s) = \int_s^t u'(r)\, dr \quad(s,t\in \R)\]
  and $a$ is bounded, we find
  \begin{align*}
    \norm{a}_\infty\abs{u(t)-u(s)} & \geq \norm{a}_\infty \abs{\int_s^t u'(r)\,dr} = \norm{a}_\infty \int_s^t \abs{u'(r)}\,dr\\
    & \geq \int_s^t \abs{(au')(r)}\, dr \geq \int_s^t C \norm{u|_I}_\infty \, dr = C \norm{u|_I}_\infty (t-s) \quad(s,t\in I).
  \end{align*}
  But trivially $\abs{u(t)-u(s)} \leq 2\norm{u|_I}_\infty$ for all $s,t\in I$, so we end up with a contradiction
  for all $C> C_0:=\max\set{\frac{2\norm{a}_\infty}{r}, \frac{\norm{\mu}_\unif}{2}}$.
  Thus, there exists $s\in I$ such that $\abs{(au')(s)}\leq C_0 \norm{u|_I}_\infty$.
  Now, for $t\in I$ we have
  \[(au')(t) = (au')(s) + \int_s^t u\, d\mu,\]
  hence
  \[\abs{(au')(t)} \leq C_0 \norm{u|_I}_\infty + \norm{\mu}_\unif \lceil r \rceil \norm{u|_I}_\infty \quad(t\in I). \qedhere\]
\end{proof}

We end this section by stating a first growth bound for solutions.

\begin{lemma}
\label{lem:SL_est1}
  Let $(a,\mu)\in \MR$, $u$ a solution of $H u = 0$. Then
  \[\abs{u(t)} + \abs{(au')(t)} \leq \bigl(\abs{u(0)} + \abs{(au')(0)}\bigr)e^{(\norm{\tfrac{1}{a}}_\infty+\norm{\mu}_\unif)(\abs {t}+1)} \quad(t\in\R).\]
\end{lemma}

\begin{proof}
  Writing
  \begin{align*}
    u(t) & = u(0) + \int_0^t u'(s)\,ds,\\
    (au')(t) & = (au')(0) + \int_0^t u(s)\, d\mu(s),
  \end{align*}
  we obtain for $\varphi(t):= \abs{u(t)} + \abs{(au')(t)}$ and $\nu:=\norm{\tfrac{1}{a}}_\infty\lambda + \abs{\mu}$ the inequality
  \[\varphi(t)\leq \varphi(0) + \int_{(t,0]} \varphi(s)\, d\nu(s) \quad(t\leq 0).\]
  By Gronwall's inequality (see Lemma \ref{lem:Gronwall_cont}) we infer
  \[\varphi(t)\leq \varphi(0)e^{\nu((t,0])} \quad(t\leq 0).\]
  Since $\norm{\nu}_\unif\leq \norm{\tfrac{1}{a}}_\infty + \norm{\mu}_\unif$ and $\nu((t,0]) \leq \norm{\nu}_\unif(\abs{t}+1)$, we obtain the assertion for $t\leq 0$.
  
  For $t>0$ we set
  \[\varphi_-(s):= \abs{u(s)} + \abs{(au')(s\llim)} \leq \varphi(0) + \int_{(0,s)} \varphi_-(r)\,d\nu(r).\]
  The Gronwall's inequality in Lemma \ref{lem:Gronwall_cont} yields
  \[\abs{u(s)} + \abs{(au')(s\llim)} = \varphi_-(s) \leq \varphi(0)e^{\nu((0,s))} = \bigl(\abs{u(0)} + \abs{(au')(0)}\bigr)e^{\nu((0,s))}.\]
  For $s\downarrow t$ the assertion follows, since $\nu((0,t])\leq \norm{\nu}_\unif(\abs{t}+1)$.
\end{proof}

\section{Estimates on differences of solutions}
\label{sec:differences_solutions}

First, we introduce Wasserstein-type seminorms on $\M$ which we will later use to measure distances of potentials.
\begin{definition}
For $\mu \in \M$ and a set $I \subseteq \R$ (which will usually be an interval) we define
\[
  \norm{\mu}_{I} := \sup\set{\Bigl|\int u\, d\mu\Bigr|;\;
  u\in W_{1,\loc}^1(\R),\: \spt u \subseteq I,\: \diam\spt u \leq 2,\: \norm{u'}_\infty \le 1 }.
\]
\end{definition}

For $\mu\in\M$ we define $\phi_\mu\from\R\to\C$ by
\[
  \phi_\mu(t) := \int_0^t d\mu
  = \begin{cases}
     \mu\bigl((0,t]\bigr) & \text{ if } t\ge0, \\
    -\mu\bigl((t,0]\bigr) & \text{ if } t<0.
    \end{cases}
\]

\begin{proposition}[see {\cite[Proposition 2.7, Remark 2.8 and Lemma 2.9]{SeifertVogt2014}}]
\label{prop:distance_measures}
  Let $\mu\in\M$ and $t\in\R$. Then
  \[\norm{\mu}_{[t-1,t+1]} \leq \min_{c\in\C} \int_{t-1}^{t+1} \abs{\phi_\mu(s)-c}\,ds \leq 2\norm{\mu}_{[t-1,t+1]}.\]
  Hence, there exists $c_{\mu,t}\in\C$, such that 
  \[\int_{t-1}^{t+1} \abs{\phi_\mu(s)-c_{\mu,t}}\,ds \leq 2\norm{\mu}_{[t-1,t+1]}.\]
  Moreover, $c_{\mu,0}$ can be chosen such that $\abs{c_{\mu,0}}\leq \norm{\mu}_\unif$.
  Furthermore, for $\alpha,\beta\in\Z$, $\alpha\leq -1$, $\beta\geq1$ and $k\in\Z\cap[\alpha,\beta-1]$ we have
  \[\int_k^{k+1} \abs{\varphi_\mu(s)-c_{\mu,0}}\,ds\leq 2\max\{k+1,-k\}\norm{\mu}_{[\alpha,\beta]}.\]
\end{proposition}

We will write $c_{\mu}:=c_{\mu,0}$.

Now, we want to estimate the difference of solutions in terms of the difference of the coefficients. 
For the diffusion coefficient we will use an $L_1$-difference, while for the potential we use the Wasserstein-type seminorm introduced above.
We will need two lemmas to describe the difference of solutions appropriately before we can state the estimate.

\begin{lemma}
\label{lem:SL-diff1}
  Let $(a,\mu), (\tilde{a},\tilde{\mu}) \in \MR$,
  $u$ and $\tilde{u}$ solutions of
  $H_{a,\mu} u = 0$ and $H_{\tilde{a},\tilde{\mu}}\tilde{u} = 0$, respectively.
  Then, for $s,t\in\R$ we have
  \begin{align*}
    \begin{pmatrix} u(t)-\tilde{u}(t)\\(au')(t) - (\tilde{a} \tilde{u}')(t)\end{pmatrix}
    & = T_{a,\mu}(t,s) \begin{pmatrix} u(s)-\tilde{u}(s)\\ (a u')(s) -  (\tilde{a}\tilde{u}')(s)\end{pmatrix} 
    + \int_s^t T_{a, \mu}(t,r)\begin{pmatrix}0 \\\tilde{u}(r)\end{pmatrix} \,d(\mu-\tilde{\mu})(r) \\
    & \quad + \int_s^t T_{a,\mu}(t,r) \begin{pmatrix} (\tilde{a}\tilde{u}')(r)\\ 0\end{pmatrix} \Bigl(\frac{1}{a(r)} - \frac{1}{\tilde{a}(r)}\Bigr)\,dr.
  \end{align*}  
\end{lemma}

\begin{proof}
  Without loss of generality, let $s=0$. Note that $-\partial_\mu (au') = u$. Integrating by parts and using the jump heights formula, we obtain
  \begin{align*}
    & \int_0^t T_{a,\mu}(r,0)^{-1}\begin{pmatrix} 0 \\ \tilde{u}(r)\end{pmatrix}\, d(\mu-\tilde{\mu})(r)
    = \begin{pmatrix} 
	- \int_0^t \uD(r)\tilde{u}(r)\,d(\mu-\tilde{\mu})(r) \\
       \int_0^t \uN(r)\tilde{u}(r)\,d(\mu-\tilde{\mu})(r)
      \end{pmatrix} \\    
    & = \begin{pmatrix} \tilde{u}(0)\\ (\tilde{a}\tilde{u}')(0)\end{pmatrix} - T_{a,\mu}(t,0)^{-1} \!\begin{pmatrix} \tilde{u}(t) \\ (\tilde{a}\tilde{u}')(t)\end{pmatrix} 
    - \int_0^t T_{a,\mu}(r,0)^{-1} \begin{pmatrix} (\tilde{a} \tilde{u}')(r)\\ 0\end{pmatrix} \Bigl(\frac{1}{a(r)} - \frac{1}{\tilde{a}(r)}\Bigr) \,dr.
  \end{align*}
  Multiplying by $T_{a,\mu}(t,0)$ yields the assertion, since we have $T_{a,\mu}(t,0) T_{a,\mu}(r,0)^{-1} = T_{a,\mu}(t,r)$
  and
  \[T_{a,\mu}(t,0) \begin{pmatrix} \tilde{u}(0)\\ (\tilde{a}\tilde{u}')(0)\end{pmatrix} = \begin{pmatrix} u(t)\\ (au')(t)\end{pmatrix} - T_{a,\mu}(t,0)\begin{pmatrix} u(0)-\tilde{u}(0)\\ (au')(0) - (\tilde{a} \tilde{u}')(0)\end{pmatrix}. \qedhere\]
\end{proof}

\begin{lemma}
\label{lem:SL_diff_aux}
  Let $(a,\mu), (\tilde{a},\tilde{\mu}) \in \MR$, $c\in\R$,
  $u$ and $\tilde{u}$ solutions of
  $H_{a,\mu} u = 0$ and $H_{\tilde{a},\tilde{\mu}}\tilde{u} = 0$, respectively, such that
  $u(0) = \tilde{u}(0)$, $(\tilde{a}\tilde{u}')(0) = (au')(0) + cu(0)$. Then
  \[u(t)-\tilde{u}(t) = \int_0^t \frac{d}{ds} \bigl(\uD(t;s)\tilde{u}(s)\bigr) \cdot \bigl(c-\varphi_{\mu-\tilde{\mu}}(s)\bigr)\,ds + \int_0^t \uN(t;s)(\tilde{a}\tilde{u}')(s) \Bigl(\tfrac{1}{a(s)} - \tfrac{1}{\tilde{a}(s)}\Bigr)\,ds\quad(t\in\R).\]
\end{lemma}

\begin{proof}
  Let $t\in\R$. By Lemma \ref{lem:SL-diff1} we obtain
  \begin{align*}
    u(t)-\tilde{u}(t) & = -c\uD(t;0)u(0) + \int_0^t \uD(t;r)\tilde{u}(r)\,d(\mu-\tilde{\mu})(r) + \int_0^t \uN(t;s)(\tilde{a}\tilde{u}')(s) \Bigl(\tfrac{1}{a(s)} - \tfrac{1}{\tilde{a}(s)}\Bigr)\,ds.
  \end{align*}
  Since $\uD(t;t) = 0$, we have
  \[\uD(t;r)\tilde{u}(r) = -\int_r^t \frac{d}{ds} \bigl(\uD(t;s)\tilde{u}(s)\bigr)\,ds.\]
  Thus, by Fubini's theorem, we obtain
  \begin{align*}
    u(t)-\tilde{u}(t) & = -c\uD(t;0)u(0) - \int_0^t \int_0^s \,d(\mu-\tilde{\mu})(r) \frac{d}{ds} \bigl(\uD(t;s)\tilde{u}(s)\bigr)\,ds \\
    & \quad + \int_0^t \uN(t;s)(\tilde{a}\tilde{u}')(s) \Bigl(\tfrac{1}{a(s)} - \tfrac{1}{\tilde{a}(s)}\Bigr)\,ds \\
    & = \int_0^t \bigl(c-\varphi_{\mu-\tilde{\mu}}(s)\bigr) \frac{d}{ds} \bigl(\uD(t;s)\tilde{u}(s)\bigr)\,ds  + \int_0^t \uN(t;s)(\tilde{a}\tilde{u}')(s) \Bigl(\tfrac{1}{a(s)} - \tfrac{1}{\tilde{a}(s)}\Bigr)\,ds. \qedhere
  \end{align*}
\end{proof}

We can now state the estimate of differences of solutions in terms of the differences of the coefficients.

\begin{lemma}
\label{lem:SL_diff2}
  Let $(a,\mu), (\tilde{a},\tilde{\mu}) \in \MR$, $u$ and $\tilde{u}$ solutions of
  $H_{a,\mu} u = 0$ and $H_{\tilde{a},\tilde{\mu}}\tilde{u} = 0$, respectively, satisfying
  \[\begin{pmatrix}
      u(0)\\
      (au')(0)
    \end{pmatrix} 
    =
    \begin{pmatrix}
      \tilde{u}(0)\\
      (\tilde{a}\tilde{u}')(0)
    \end{pmatrix}
    - 
    \begin{pmatrix}
      0\\
      c_{\mu-\tilde{\mu}} \tilde{u}(0)
    \end{pmatrix}.\]    
  Let $\alpha,\beta\in\Z$, $\alpha\leq -1$, $\beta\geq 1$. Let $c,\omega>0$ such that
  \[\abs{\uN(t;s)}, \abs{\bigl(a\uD'(\cdot;s)\bigr)(t)} \leq ce^{\omega\abs{t-s}} \quad(s,t\in\R).\]
  Then there exists a constant $C>0$ depending only on $\omega$ and $\norm{\frac{1}{a}}_\infty, \norm{\mu}_\unif$, $\norm{\frac{1}{\tilde{a}}}_\infty, \norm{\tilde{\mu}}_\unif$ such that
  \[\abs{u(t) - \tilde{u}(t)} \leq Cce^{\omega \abs{t}} 
  \norm{\tilde{u}|_{[\alpha,\beta]}}_\infty \bigl(\int_{\alpha}^\beta \abs{a(s) - \tilde{a}(s)}\, ds + \norm{\mu-\tilde{\mu}}_{[\alpha,\beta]}\bigr) \quad(t\in[\alpha,\beta]).\]
\end{lemma}

\begin{proof}
  From $\uD(s;s) = 0$ and the assumed bound $\abs{(a\uD')(\cdot;s)(t)} \leq ce^{\omega\abs{t-s}}$ for all $s,t\in\R$ we obtain
  \[\abs{\uD(t;s)} = \abs{\int_s^t \uD'(r;s)\,dr} = \abs{\int_s^t \frac{1}{a(r)} \bigl(a\uD'(\cdot;s)\bigr)(r)\,dr} \leq \norm{\tfrac{1}{a}}_\infty \frac{c}{\omega} e^{\omega\abs{t-s}} \quad (s,t\in\R).\]
  By Lemma \ref{lem:estimate_derivative} (with $r=1$) we have
  \[\norm{\bigl(\tilde{a}\tilde{u}'(\cdot;s)\bigr)|_{[\alpha,\beta]}}_\infty \leq \Bigl(\max\set{2\norm{\tilde{\alpha}}_\infty, \frac{\norm{\tilde{\mu}}_\unif}{2}} + \norm{\tilde{\mu}}_\unif\Bigr)\norm{\tilde{u}|_{[\alpha,\beta]}}_\infty.\]
  Since $\uD(t;s) = -\uD(s;t)$, we obtain
  \begin{align*}
    & \abs{\frac{d}{ds} \bigl(\uD(t;s)\tilde{u}(s)\bigr)} = \abs{-\uD'(s;t)\tilde{u}(s) + \uD(t;s)\tilde{u}'(s)}\\
    & \leq \norm{\tfrac{1}{a}}_\infty ce^{\omega\abs{t-s}} \norm{\tilde{u}|_{[\alpha,\beta]}}_\infty + \norm{\tfrac{1}{a}}_\infty \frac{c}{\omega} e^{\omega\abs{t-s}} \norm{\tfrac{1}{\tilde{a}}}_\infty\bigl(\max\set{2\norm{\tilde{\alpha}}_\infty, \tfrac{\norm{\tilde{\mu}}_\unif}{2}} + \norm{\tilde{\mu}}_\unif\bigr)\norm{\tilde{u}|_{[\alpha,\beta]}}_\infty\\
    & = C_0 c \norm{\tilde{u}|_{[\alpha,\beta]}}_\infty e^{\omega\abs{t-s}} \quad(s,t\in [\alpha,\beta]),
  \end{align*}
  for some $C_0\geq 0$.
  
  Let $t\in [0,\beta]$. By Lemma \ref{lem:SL_diff_aux} and Proposition \ref{prop:distance_measures} we have
  \begin{align*}
    \abs{u(t)-\tilde{u}(t)} & \leq \int_0^t \abs{\frac{d}{ds} \bigl(\uD(t;s)\tilde{u}(s)\bigr)} \abs{c_{\mu-\tilde{\mu}}-\varphi_{\mu-\tilde{\mu}}(s)}\,ds + \int_0^t \abs{\uN(t;s)}\abs{(\tilde{a}\tilde{u}')(s)} \abs{\tfrac{1}{a(s)} - \tfrac{1}{\tilde{a}(s)}}\,ds \\
    & \leq C_0 c \norm{\tilde{u}|_{[\alpha,\beta]}}_\infty \sum_{k=1}^\beta \int_{k-1}^k e^{\omega(t-s)} \abs{c_{\mu-\tilde{\mu}}-\varphi_{\mu-\tilde{\mu}}(s)}\,ds \\
    & \quad + ce^{\omega t} \Bigl(\max\set{2\norm{\tilde{\alpha}}_\infty, \frac{\norm{\tilde{\mu}}_\unif}{2}} + \norm{\tilde{\mu}}_\unif\Bigr)\norm{\tilde{u}|_{[\alpha,\beta]}}_\infty \norm{\tfrac{1}{a}}_\infty \norm{\tfrac{1}{\tilde{a}}}_\infty\int_0^t \abs{\tilde{a}(s) - a(s)}\,ds\\
    & \leq C_1 c e^{\omega t}\norm{\tilde{u}|_{[\alpha,\beta]}}_\infty \Bigl(\norm{a-\tilde{a}}_{L_1(\alpha,\beta)} + \norm{\mu-\tilde{\mu}}_{[\alpha,\beta]}\Bigr),
  \end{align*}
  for some $C_1\geq 0$. The proof in the case $t\in[\alpha,0)$ is analogous.
\end{proof}

By making use of this estimate we can now improve the growth bound of solutions obtained in Lemma \ref{lem:SL_est1}.

\begin{lemma}
\label{lem:SL_est2}
  Let $(a,\mu)\in \MR$, $u$ a solution of
  $H_{a,\mu} u = 0$, $\omega:= \bigl(\norm{\mu}_\unif\norm{\frac{1}{a}}_\infty^{-1}\bigr)^{1/2}$.
  Then
  \[\bigl(\omega^2 \abs{u(t)}^2 + \abs{(au')(t)}^2\bigr)^{1/2} \leq \bigl(\omega^2 \abs{u(0)}^2 + \abs{(au')(0)}^2\bigr)^{1/2} e^{\omega \norm{\tfrac{1}{a}}_\infty (\abs{t}+1/2)} \quad(t\in\R).\]
\end{lemma}

\begin{proof}
  Without loss of generality, let $\mu\neq 0$ (the case $\mu=0$ is trivial, as then $(au')$ is constant).
  
  (i) 
  We first assume that $\mu = q\lambda$ with a density $q\in C(\R)$. Then $au'\in C^1(\R)$ and $(au')' = qu$. Let $\phi(t):= \omega^2\abs{u(t)}^2 + \abs{(au')(t)}^2$. Then $\varphi\in W_{1,\loc}^1(\R)$, and
  \begin{align*}
  	\abs{\phi'(t)} & = \abs{2\Re\bigl(\bigl(\tfrac{\omega}{\overline{a(t)}} + \abs{q(t)}\bigr)u(t) \overline{(au')(t)}\bigr)} \leq \bigl(\tfrac{\omega}{\abs{a(t)}} + \tfrac{\abs{q(t)}}{\omega}\bigr)\varphi(t)
  \end{align*}
  for a.a.\ $t\in\R$.  
  Hence, $\varphi(t)\leq \varphi(s)\exp(\omega\norm{\frac{1}{a}}_\infty\abs{t-s} + \frac{1}{\omega}\int_s^t\rho(r)\,dr)$ and therefore
  \[\bigl(\omega^2 \abs{u(t)}^2 + \abs{(au')(t)}^2\bigr)\leq \bigl(\omega^2 \abs{u(s)}^2 + \abs{(au')(s)}^2\bigr)e^{\omega\norm{\tfrac{1}{a}}_\infty(t-s) + \frac{1}{\omega}\abs{\mu}([s,t])}\]
  for all $s,t\in\R$, $s<t$.
  
  (ii)
  By \cite[Proposition 2.5]{SeifertVogt2014} there exists $(\mu_n)$ in $\M$ such that $\mu_n$ has a smooth density and $\norm{\mu_n}_\unif\leq \norm{\mu}_\unif$ for all $n\in\N$, $\norm{\mu_n-\mu}_\R\to 0$ and $\limsup_{n\to\infty} \abs{\mu_n}(I) \leq \abs{\mu}(I)$ for all compact intervals $I\subseteq \R$.
  Then \cite[Lemma 2.4]{SeifertVogt2014} implies $\1_{[\alpha,\beta]}\mu_n\to \1_{[\alpha,\beta]}\mu$ weakly for all $\alpha,\beta\in\R$ such that $\mu(\{\alpha\}) = \mu(\{\beta\}) = 0$.
  
  (iii)
  For $n\in\N$ let $u_n$ be the solution of $H_{a,\mu_n}u_n = 0$ such that 
  $u_n(0) = u(0)$, $(au_n)'(0) = (au')(0) + c_{\mu-\mu_n}u(0)$.
  By Lemma \ref{lem:SL_est1}, $(u_n)$ is uniformly bounded on any compact interval, so Lemma \ref{lem:SL_diff2} implies $u_n\to u$ locally uniformly. Hence, for $s,t\in\R$ with $\mu(\{s\}) = \mu(\{t\}) = 0$ we obtain
  \[(au_n')(t) - (au_n')(s) = \int_s^t u_n(r) \,d\mu_n(r) \to \int_s^t u(r)\,d\mu(r) = (au')(t) - (au')(s).\]
  By Lemma \ref{lem:estimate_derivative} also $(au_n')$ is uniformly bounded on $[0,1]$, so dividing by $a(s)$ and integration with respect to $s$ yields
  \begin{align*}
    (au_n')(t)\int_0^1 \frac{1}{a(s)}\,ds  - \bigl(u_n(1) - u_n(0)\bigr) & \to (au')(t)\int_0^1 \frac{1}{a(s)}\,ds  - \bigl(u(1) - u(0)\bigr),
  \end{align*}
  so $(au_n')(t) \to (au')(t)$.
  
  (iv)
  Let $t>s>0$ such that $\mu(\{s\}) = \mu(\{t\}) = 0$. By (i) we have
  \[\bigl(\omega^2 \abs{u_n(t)}^2 + \abs{(au_n')(t)}^2\bigr)\leq \bigl(\omega^2 \abs{u_n(s)}^2 + \abs{(au_n')(s)}^2\bigr)e^{\omega\norm{\tfrac{1}{a}}_\infty(t-s) + \frac{1}{\omega}\abs{\mu_n}([s,t])}.\]
  Taking the limit $n\to\infty$ noting (ii) we obtain
  \[\bigl(\omega^2 \abs{u(t)}^2 + \abs{(au')(t)}^2\bigr)\leq \bigl(\omega^2 \abs{u(s)}^2 + \abs{(au')(s)}^2\bigr)e^{\omega\norm{\tfrac{1}{a}}_\infty(t-s) + \frac{1}{\omega}\abs{\mu}([s,t])}.\]
  
  (v)
  For $t>0$ there exist sequences $s_n \in [0,t)$ and $(t_n)$ in $[t,\infty)$ 
  such that $s_n\to 0$, $t_n\to t$ and $\mu(\{s_n\}) = \mu(\{t_n\}) = 0$ for 
  all $n\in\N$.
  Thus, from (iv) we deduce
  \[\bigl(\omega^2 \abs{u(t)}^2 + \abs{(au')(t)}^2\bigr)\leq \bigl(\omega^2 \abs{u(0)}^2 + \abs{(au')(0)}^2\bigr)e^{\omega\norm{\tfrac{1}{a}}_\infty t + \frac{1}{\omega}\abs{\mu}((0,t])}.\]
  Hence,
  \[\bigl(\omega^2 \abs{u(t)}^2 + \abs{(au')(t)}^2\bigr)\leq \bigl(\omega^2 \abs{u(0)}^2 + \abs{(au')(0)}^2\bigr)e^{\omega\norm{\tfrac{1}{a}}_\infty t + \frac{1}{\omega}\norm{\mu}_{\unif}(t+1)}.\]  
  Optimizing for $\omega>0$ yields $\omega = 
  \bigl(\norm{\mu}_\unif\norm{\frac{1}{a}}_\infty^{-1}\bigr)^{1/2}$, which implies the assertion. The case $t<0$ is proved analogously.
\end{proof}

\section{Bounds on eigenvalues}
\label{sec:Gordons_Theorem}

\begin{definition}
  Let $(a,\mu)\in\MR$, $C\geq 0$. We say that $(a,\mu)$ satisfies a \emph{weak Gordon condition} with weight $C$, provided
  there exists a sequence $(p_m)$ in $(0,\infty)$, $p_m\to \infty$, such that 
  \[\lim_{m\to\infty} e^{Cp_m} \Bigl(\norm{a-a(\cdot + p_m)}_{L_1(-p_m,p_m)} + \norm{\mu - \mu(\cdot+p_m)}_{[-p_m,p_m]}\Bigr) = 0.\]
\end{definition}

\begin{lemma}[see {\cite[Lemma 5.1]{SeifertVogt2014}}]
\label{lem:equiv_Gordon}
  Let $\mu\in\M$, $C > 0$. Assume there exists $(p_m)$ in $(0,\infty)$ with $p_m\to \infty$ such that
  \[e^{Cp_m} \norm{\mu-\mu(\cdot+p_m)}_{[-p_m,p_m]}\to 0.\]
  Then there exists $(\mu_m)$ in $\M$ such that
      $\mu_m$ is periodic with period $p_m$ ($m\in\N$), and
      \[
        e^{Cp_m} \norm{\mu-\mu_m}_{[-p_m,2p_m]} \to 0 \qquad (m\to\infty).
      \]
  Moreover, the measures $\mu_m$ can be chosen such that
  \[
    \1_{[\alpha_m,p_m-\alpha_m]} \mu_m = \1_{[\alpha_m,p_m-\alpha_m]} \mu, \quad
    \norm{\mu_m}_\lu \le \bigl(1+\tfrac{1}{2\alpha_m}\bigr)\norm{\mu}_\lu
  \]
  for all $m\in\N$, with $0 < \alpha_m \le \frac{p_m}{2}$ and $\inf_{m\in\N} \alpha_m>0$.
\end{lemma}

\begin{lemma}
\label{lem:SL_per}
  Let $p\in \Per(\rho)$, $(a,\mu)\in \MR$ be $p$-periodic. Let $z\in\C$ and $u$ a 
  solution of $Hu = zu$. Then
  \[\max\set{\norm{\begin{pmatrix} u(t)\\ (au')(t)\end{pmatrix}};\; t\in\set{-p,p,2p}} \geq \frac{1}{2}\norm{\begin{pmatrix} u(0)\\ (au')(0)\end{pmatrix}}.\]
\end{lemma}

\begin{proof}
  Without loss of generality, let $z=0$ (just consider $\mu-z\rho$ instead of $\mu$).
  Note that $T:= T_{a,\mu}(p,0) = T_{a,\mu}(2p,p) = T_{a,\mu}(0,-p)$ by periodicitiy.
  The Cayley-Hamliton theorem assures that (note that $\det T=1$)
  \[T^2 -\tr(T) T + I = 0.\]
  Applying this equality to $\bigl(u(-p),(au')(-p)\bigr)$ in case $\abs{\tr(T)}\leq 1$ and to $\bigl(u(0),(au')(0)\bigr)$ in case $\abs{\tr(T)}>1$ yields the assertion.
\end{proof}

We can now state Gordon's theorem for Sturm-Liouville operators with measure-valued coefficients.

\begin{theorem}
\label{thm:SL}
  Let $\rho$ be periodic, $(a,\mu)\in\MR$ satisfy the weak Gordon condition with $C>0$ with period sequence $(p_m)$ in $\Per(\rho)$.
  Then $H$ does not have any eigenvalues with modulus less than 
  $\frac{1}{\norm{\rho}_\unif}\bigl(\norm{\tfrac{1}{a}}_\infty^{-1} C^2 - \norm{\mu}_\unif\bigr)$.
\end{theorem}

\begin{proof}
  Let $(\mu_m)$ as in Lemma \ref{lem:equiv_Gordon}.
  Without loss of generality, let $p_m\geq 2$ for all $m\in\N$, and we may further assume that $p_m+\alpha_m\in\N$ for all $m\in\N$, $\alpha_m\to \infty$ and $\frac{\alpha_m}{p_m}\to 0$.
  For $m\in\N$ let $a_m$ be $p_m$-periodic with $a_m|_{(0,p_m]} = a|_{(0,p_m]}$.
   
  Assume that $z\in\R$ with
  $\abs{z}< \frac{1}{\norm{\rho}_\unif}\bigl(\norm{\tfrac{1}{a}}_\infty^{-1} C^2 - \norm{\mu}_\unif\bigr)$
  is an eigenvalue of $H=H_{p,\rho,a,\mu}$.
  Let $u\in L_p(\R,\rho)$, $u\neq 0$ be a corresponding eigenfunction. Then $u$ is bounded, since $u$ is a solution and thus continuous,
  and tends to zero by Lemma \ref{lem:tends_to_zero}.
  For $m\in\N$ let $u_m$ be the solution of $H_{p,\rho,a_m,\mu_m} u_m = zu_m$ satisfying $u_m(\alpha_m) = u(\alpha_m)$, $(a_m u_m')(\alpha_m) = (a u')(\alpha_m)$.
  Then $u_m = u$ on $[\alpha_m,p_m-\alpha_m]$, since $a_m=a$ and $\mu_m = \mu$ on this interval. Note that $c_{\mu-\tilde{\mu},\alpha_m+1} = 0$, since $\1_{[\alpha_m,\alpha_m+2]}(\mu_m-\mu) = 0$.
  By Lemma \ref{lem:SL_diff2}, for $t\in[-p_m,\alpha_m]$ we obtain
  \[\abs{u(t) - u_m(t)} \leq C_m 
  e^{\omega_m\abs{t-(\alpha_m+1)}}\bigl(\norm{a-a_m}_{L_1(-p_m,\alpha_m+1)} + 
  \norm{\mu-\mu_m}_{[-p_m,\alpha_m+1]}\bigr)\]
  where 
  $\omega_m = \norm{\tfrac{1}{a_m}}_\infty^{1/2}\norm{\mu_m-z\rho}_\unif^{1/2}$ as in Lemma \ref{lem:SL_est2}, 
  and $C_m$ is only depending on $\omega_m$, $\norm{\frac{1}{a}}_\infty$, $\norm{\frac{1}{a_m}}_\infty$, $\norm{\mu}_\unif$ and $\norm{a}_\infty$, and similarly for $t\in [p_m-\alpha_m,2p_m]$.
  Hence,
  \begin{align*}
  \label{eq:SL_conv}
    \sup_{t\in[-p_m,2p_m]} \abs{u(t) - u_m(t)}
    & \leq C_m e^{\omega_m (p_m+\alpha_m+1)}\bigl(\norm{a-a_m}_{L_1(-p_m,2p_m)}
    + \norm{\mu-\mu_m}_{[-p_m,2p_m]}\bigr).
  \end{align*}
  Since $\norm{\frac{1}{a_m}}_\infty \leq \norm{\frac{1}{a}}_\infty$,
  we have
  \begin{align*}
    \omega_m & \leq \norm{\frac{1}{a_m}}_\infty^{1/2}\norm{\mu_m-z\rho}_\unif^{1/2}
    \leq \norm{\frac{1}{a}}_\infty^{1/2}\bigl(\norm{\mu_m}_\unif+\abs{z}\norm{\rho}_\unif\bigr)^{1/2} \\
    & \leq \norm{\frac{1}{a}}_\infty^{1/2}\bigl((1+\tfrac{1}{2\alpha_m})\norm{\mu}_\unif+\abs{z}\norm{\rho}_\unif\bigr)^{1/2} \\
    & \to \norm{\frac{1}{a}}_\infty^{1/2} \bigl(\norm{\mu}_\unif + \abs{z}\norm{\rho}_\unif\bigr)^{1/2} < C,
  \end{align*} 
  so for large $m$ we obtain
  \[\omega_m(p_m+\alpha_m+1) \leq Cp_m.\]
  Thus, for $\varepsilon>0$ there exists $m_0\in\N$ such that such that $\abs{u(t)-u_m(t)}\leq \varepsilon$ for all $m\geq m_0$ and $t\in[-p_m,2p_m]$.
  By Lemma \ref{lem:conv_to_0} there exists $m_1\geq m_0$ such that $\abs{u(t)}\leq \varepsilon$ for $\abs{t}\geq p_{m_1}-1=:t_1$. 
  Then $\abs{u_m}\leq 2\varepsilon$ on $[-p_m,2p_m]\setminus (-t_1,t_1)$, for all $m\geq m_1$.
  By Lemma \ref{lem:estimate_derivative} we obtain $\abs{a_mu_m'}\leq C_{1,a_m,\mu_m} 2\varepsilon$ on that set.
  Hence,
  \[\bigl(u_m(\pm p_m), (a_mu_m')(\pm p_m)\bigr),\bigl(u_m(2p_m), (a_mu_m')(2 p_m)\bigr) \to 0 \quad(m\to\infty).\]
  Lemma \ref{lem:SL_per} yields $\bigl(u_m(0), (a_mu_m')(0)\bigr)\to 0$. 
  By Lemma \ref{lem:SL_est1} we now obtain $u_m\to 0$ locally uniformly. Since 
  $u_m\to u$ locally uniformly by Lemma \ref{lem:SL_diff2}, we obtain $u=0$, a 
  contradiction.
\end{proof}

By applying Theorem \ref{thm:SL} for arbitrarily large $C>0$ we obtain absence of eigenvalues.

\begin{corollary}
\label{cor:absence_ev}
  Assume $(a,\mu)\in\MR$ satisfy the weak Gordon condition for all $C>0$ with period sequence $(p_m)$ in $\Per(\rho)$. Then $H$ does not have any eigenvalues.
\end{corollary}

\begin{remark}
  The proof of Theorem \ref{thm:SL} actually shows that $Hu=zu$ does not have any solution in $C_0(\R)$ for $z$ with small modulus.
\end{remark}

\begin{remark}
  The obtained bound is in general not optimal, but in some sense close to optimal, which we will make precise now.
  
  \begin{enumerate}
    \item
      For $r>0$ define
      \[\norm{\mu}_{\unif,r}:=\frac{1}{r} \sup_{t\in\R} \abs{\mu}\bigl(t,t+r]).\]
      Note that $\norm{\mu}_{\unif,1} = \norm{\mu}_\unif$.
      A scaling argument yields the following:
      Let $\rho$ be periodic, $(a,\mu)\in\MR$ satisfy the weak Gordon condition with $C>0$ with period sequence $(p_m)$ in $\Per(\rho)$.
      Then $H$ does not have any eigenvalues with modulus less than 
      \[\inf_{r>0} \frac{1}{\norm{\rho}_{\unif,r}}\bigl(\norm{\tfrac{1}{a}}_\infty^{-1} C^2 - \norm{\mu}_{\unif,r}\bigr).\]
    \item
      Let $\rho$ be periodic, $(a,\mu)\in\MR$. Then the supremum of all $C>0$ such that $(a,\mu)$ satisfies the weak Gordon condition with $C_{(a,\mu)}>0$ is given by
      \[ C_{(a,\mu)}:= - \liminf_{p\to\infty} \frac{1}{p} \log \Bigl(\norm{a-a(\cdot + p)}_{L_1(-p,p)} + \norm{\mu - \mu(\cdot+p)}_{[-p,p]}\Bigr),\]
      whenever $C_{(a,\mu)}>0$.
    \item
      Let $\rho$ be periodic, $(a,\mu)\in\MR$ satisfy the weak Gordon condition with $C_{(a,\mu)}>0$ with period sequence $(p_m)$ in $\Per(\rho)$.
      Then one can show that $H$ does not have any eigenvalues with modulus less than
      \[\inf_{r>0} \frac{1}{\norm{\rho}_{\unif,r}}\bigl(\norm{\tfrac{1}{a}}_\infty^{-1} C_{(a,\mu)}^2 - \norm{\mu}_{\unif,r}\bigr).\]
    \item
      The bound given in (c) is sharp in the continuum Schr\"odinger case, see \cite[Section 6]{SeifertVogt2014}.
      The example constructed there generalizes to our situation without any difficulty.
  \end{enumerate}
\end{remark}

\begin{example}
  Typical examples for coefficients satisfying our weak Gordon condition are constructed by sums of periodic ones, where the ratio of ther periods is  
  an irrational number, which can be superexponentially fast approximated by rational numbers.
  Without loss of generality, let $1\in\Per(\rho)$.
  Let $\alpha\in(0,\infty)$ be irrational and satisfy
  \[\abs{\alpha-\frac{p_m}{q_m}} \leq B m^{-q_m} \quad(m\in\N)\]
  for some $B>0$ and a suitable sequence $(\frac{p_m}{q_m})$ in $\Q$.
  Note that the set of all such numbers $\alpha$ is a dense $G_\delta$ set.
  
  Let $(a_1,\mu_1)\in\MR$ be $1$-periodic, $(a_2,\mu_2)\in\MR$ be $\alpha$-periodic, where $a_2$ is H\"older-continuous with exponent $\beta>0$, i.e.\ there exists $c>0$, such that
  \[\abs{a_2(x) - a_2(y)} \leq c\abs{x-y}^\beta \quad(x,y\in\R),\]
  and consider $(a,\mu):=(a_1+a_2,\mu_1+\mu_2)\in\MR$.
  Then these coefficients are quasiperiodic. Let $C>0$. Then
  \begin{align*}
    e^{Cp_m} \norm{a-a(\cdot+p_m)}_{L_1(-p_m,p_m)} & = e^{Cp_m}\norm{a_2 - a_2(\cdot+p_m)}_{L_1(-p_m,p_m)} \\
    & = e^{Cp_m}\norm{a_2 - a_2(\cdot+p_m-\alpha q_m)}_{L_1(-p_m,p_m)} \leq e^{Cp_m} 2p_m c\abs{p_m-\alpha q_m}^\beta \\
    & \leq 2c e^{Cp_m} p_m q_m \abs{\alpha-\frac{p_m}{q_m}}^\beta \leq 2c e^{Cp_m} B p_mq_m m^{-q_m} \to 0.
  \end{align*}
  Furthermore, as translation of $\mu_2$ is Lipschitz continuous for the norm $\norm{\cdot}_\R$ with Lipschitz constant $3\norm{\mu_2}_\unif$, we obtain
  \begin{align*}
    e^{Cp_m}\norm{\mu - \mu(\cdot+p_m)}_{[-p_m,p_m]} & = e^{Cp_m}\norm{\mu_2 - \mu_2(\cdot+p_m)}_{[-p_m,p_m]} \\
    & = e^{Cp_m}\norm{\mu_2 - \mu_2(\cdot+p_m-\alpha q_m)}_{[-p_m,p_m]} \leq e^{Cp_m} 3\abs{p_m - \alpha q_m} \norm{\mu_2}_\unif \\
    & \leq 3 \norm{\mu_2}_\unif e^{Cp_m} q_m \abs{\alpha - \frac{p_m}{q_m}} \leq  3 \norm{\mu_2}_\unif e^{Cp_m} q_m B m^{-q_m} \to 0.
  \end{align*}
  Thus, $(a,\mu)$ satisfies a Gordon condition for all $C>0$, so Corollary \ref{cor:absence_ev} yields absence of eigenvalues for $H_{p,\rho,a,\mu}$.  
\end{example}

\appendix

\section{Gronwall inequality}

We provide a Gronwall inequality suitable for our context. We include the proof for the reader's convenience.

\begin{lemma}
\label{lem:Gronwall_cont}
  Let $\alpha\from[0,\infty)\to [0,\infty)$ be measurable, $\mu$ a nonnegative Borel measure on $[0,\infty)$ and $u\in \mathcal{L}_{1,\loc}([0,\infty),\mu)$ such that
  \[u(t) \leq \alpha(t) + \int_{[0,t)} u(s)\,d\mu(s) \quad(t\geq 0).\]
  Then
  \[u(t) \leq \alpha(t) + \int_{[0,t)} \alpha(s)\exp\bigl(\mu\bigl((s,t)\bigr)\bigr)\,d\mu(s) \quad(t\geq 0).\]
\end{lemma}

\begin{proof}
  (i) Iterating the inequality yields
  \[u(t)\leq \alpha(t) + \int_{[0,t)} \alpha(s) \sum_{k=0}^{n-1}\mu^{\otimes k}\bigl(A_k(s,t)\bigr) \, d\mu(s) + R_n(t) \quad(n\in\N, t\geq 0),\]
  where
  \[R_n(t) := \int_{[0,t)} u(s) \mu^{\otimes n}\bigl(A_n(s,t)\bigr) \, d\mu(s)\]
  is the remainder, $A_k(s,t) := \set{(s_1,\ldots,s_k)\in(s,t)^k;\;s_1<\ldots<s_k}$
  is an $k$-dimensional open simplex and
  $\mu^{\otimes 0}\bigl(A_0(s,t)\bigr) := 1$.
  
  (ii) Let $0\leq s<t$. We now prove
  \[\mu^{\otimes k}\bigl(A_k(s,t)\bigr) \leq \frac{\mu\bigl((s,t)\bigr)^k}{k!} \quad(k\in\N_0).\]
  Indeed, let $S_k$ be the set of all permutations of $\set{1,\ldots,k}$. For $\sigma\in S_k$ let
  \[A_{k,\sigma}(s,t) := \set{(s_1,\ldots,s_k)\in(s,t)^k;\; s_{\sigma(1)}<\ldots<s_{\sigma(k)}}.\]
  Then for $\sigma\neq \sigma'$ we obtain $A_{k\sigma}(s,t)\cap A_{k,\sigma'}(s,t) = \varnothing$. 
  Furthermore,
  \[\bigcup_{\sigma\in S_k} A_{k,\sigma}(s,t) \subseteq (s,t)^k.\]
  Hence,
  \[k! \mu^{\otimes k} \bigl(A_k(s,t)\bigr) = \sum_{\sigma\in S_k}\mu^{\otimes k} \bigl(A_k(s,t)\bigr) \leq \mu^{\otimes k}\bigl((s,t)^k\bigr) = \mu\bigl((s,t)\bigr)^k.\]
  
  (iii) By (ii), we obtain
  \[\abs{R_n(t)} \leq \frac{\mu\bigl((s,t)\bigr)^n}{n!} \int_{[0,t)} \abs{u(s)}\,d\mu(s) \quad(n\in\N, t\geq 0).\]
  Since $u$ is locally integrable with respect to $\mu$ we obtain $R_n\to 0$ pointwise.
  Thus, (i) yields
  \begin{align*}
    u(t) & \leq \alpha(t) + \int_{[0,t)} \alpha(s) \sum_{k=0}^{n-1}\frac{\mu\bigl((s,t)\bigr)^k}{k!} \, d\mu(s) + R_n(t) \\
    & \leq \alpha(t) + \int_{[0,t)} \alpha(s) \exp\bigl(\mu\bigl((s,t)\bigr)\bigr) \, d\mu(s) + R_n(t) \\
    & \to \alpha(t) + \int_{[0,t)} \alpha(s) \exp\bigl(\mu\bigl((s,t)\bigr)\bigr) \, d\mu(s). \qedhere
  \end{align*}
\end{proof}

\bibliographystyle{elsarticle-num}

\end{document}